\documentclass[12pt,reqno]{amsart} 

\usepackage{amsaddr} 

\usepackage{mathrsfs,dsfont} 
\usepackage{wrapfig} 
\usepackage{graphicx} 
\usepackage{tkz-berge}
\usepackage{tikz-cd}
\usepackage{changepage}
	\usetikzlibrary{positioning}
	\usepackage{pgfornament}  
	\usepackage{comment}
	\usepackage[linktoc=all]{hyperref} \hypersetup{ colorlinks, linkcolor=black, urlcolor= black } 
	\usepackage{enumerate}
\allowdisplaybreaks  
\usepackage{tikz}	
\usepackage{geometry}

	\theoremstyle{plain}
	\newtheorem{theorem}{Theorem}[section]
	\newtheorem{lem}{Lemma}[section]
	
	\newtheorem{prop}{Proposition}[section]
	\newtheorem{cor}{Corollary}[section]
	\theoremstyle{definition}
	\newtheorem{defn}{Definition}[section]
	\newtheorem*{acknowledgements}{Acknowledgements}
	\theoremstyle{remark}
	\newtheorem{rmk}{Remark}
	\newtheorem{eg}{Example}

	\newcommand{\R}{\mathbb{R}}
	\newcommand{\C}{\mathbb{C}}
	\newcommand{\N}{\mathbb{N}}
	\newcommand{\Z}{\mathbb{Z}}
	\newcommand{\F}{\mathbb{F}}

	\renewcommand{\O}{\mathcal{O}}
	
	\newcommand{\Q}{\mathbb{Q}}
	\newcommand{\B}{\mathscr{B}}

	\newcommand{\f}{\mathfrak{f}} 
	 
	\renewcommand{\dfrac}[2]{\tfrac{#1}{#2}}
	\newcommand{\ol}[1]{\overline{#1}}

	\newcommand{\sbst}{\subseteq}
	
	\newcommand{\Nm}{\text{Nm}}
	\renewcommand{\#}{\sharp}
	\renewcommand{\l}{\ell}
	\renewcommand{\ss}{supersingular}

\hypersetup{citecolor=olive }

\begin{document}

\newgeometry{margin=1in}

\title[Loops, multi-edges and collisions]{Loops, multi-edges and collisions in supersingular isogeny graphs}

\author[Wissam Ghantous]{Wissam Ghantous}
\address{Mathematical Institute,\\University of Oxford, \\ Andrew Wiles Building, OX2 6GG, UK}
\subjclass[2010]{14H52, 14K02}

\date{\today }

\keywords{supersingular isogeny graphs, loops, multi-edges, collisions, bi-route number}

\begin{abstract}
Supersingular isogeny graphs are known to have very few loops and multi-edges. 
	We formalize this idea by studying and finding bounds for the number of loops and multi-edges in such graphs. 
	We also find conditions under which the supersingular isogeny graph $\Lambda_p(\l)$ is simple.

	The methods presented in this paper can be used to study many kinds of collisions in supersingular isogeny graphs. 
	{As an application,} we introduce the notion of bi-route number for two graphs $\Lambda_p(\l_1),\Lambda_p(\l_2)$ 
	and compute bounds for it. 
	We also study the number of edges in common between the graphs $\Lambda_p(\l_1),\Lambda_p(\l_2)$. 
\end{abstract}

\maketitle

\section{Introduction}  

Supersingular isogeny graphs $\Lambda_p(\l)$ play a paramount role in post-quantum cryptography and isogeny based cryptography. They are used to develop quantum-safe analogs of older cryptographic systems. 
In \cite{CLG09}, Charles, Lauter and Goren construct a cryptographic hash function whose security is based on the difficulty of finding paths between two vertices in $\Lambda_p(\l)$, i.e. finding isogenies between supersingular elliptic curves. 
Moreover, in \cite{DFJP14} new candidates for quantum-resistant public-key cryptosystems, based on this same problem, are presented. 
The graph $\Lambda_p(\l)$ is created by taking the graph associated to a Brandt matrix $B(\l)$ for the prime $\l$, over some base prime $p$. In other words, they are obtained by turning the set of supersingular elliptic curves (over $\F_{p^{2}}$) into a graph where edges are given by isogenies of prime degree $\l$. 
\\

A key assumption that is made in many cases is that the base prime $p$ is congruent to $1$ modulo $12$. This is a very technical assumption to ensure that the graphs are everywhere regular and undirected. Actually, the only vertices where one might not have undirectedness are the ones corresponding to the $j$-invariants $0$ and $1728$. The assumption $p \equiv 1 \mod 12$ ensures that the elliptic curves with $j$-invariants $0$ and $1728$ are not supersingular. 

A further conceptual simplification that is usually made when working with supersingular isogeny graphs is that they are very close to being simple graphs. Indeed, supersingular isogeny graphs have very few multiple edges and also very few loops (by loop, we mean an edge from a vertex to itself). 
One might hence ask how many \emph{loops} or \emph{multi-edges} such a graph has on average, or what assumptions must one make in order to obtain a simple supersingular isogeny graph (i.e. with no loops nor multi-edges). 

\textbf{Our contribution.} It is already known that one can make some conditions on the prime $p$ to ensure that the graph $\Lambda_p(\l)$ has no loops. 
However there isn’t any comprehensive study assembling -- in one place -- bounds on the number of loops, bounds on the number of multi-edges and conditions to ensure the simplicity of $\Lambda_p(\l)$. 
To do so, 
we will use a very particular characterization of the trace of Brandt matrices proven in \cite{Gro87}, relating the trace of Brandt matrices to modified Hurwitz class number, as well as bounds for Hurwitz class numbers given in \cite{Gie80}.

Moreover, since the graphs $\{\Lambda_p(\l) \}_\l$ have the same set of vertices for any fixed prime $p$, 
we can consider the graph $\Lambda_p(\l_1, \l_2)$, for two primes $\l_1$ and $\l_2$, obtained by superposing the two graphs $\Lambda_p(\l_1), \Lambda_p(\l_2)$ 
and drawing the edges of $\Lambda_p(\l_1)$ and $\Lambda_p(\l_2)$ in two different colours as in Figure 1. 
One is then interested in the number of edges that these two graphs have in common (which is related to their \emph{edit distance}) as well as the number of times two vertices will have paths of certain lengths between them on both graphs (which we will define as their \emph{bi-route number}). 
We will formalize, quantitatively study and bound these notions of common edges and common paths 
between two graphs $\Lambda_p(\l_1)$ and $\Lambda_p(\l_2)$. 
We will do so by using the same method as for the study of loops and multi-edges. 
This method can actually be used to quantitatively study any kind of collisions in supersingular isogeny graphs. 
\begin{figure}[h!]
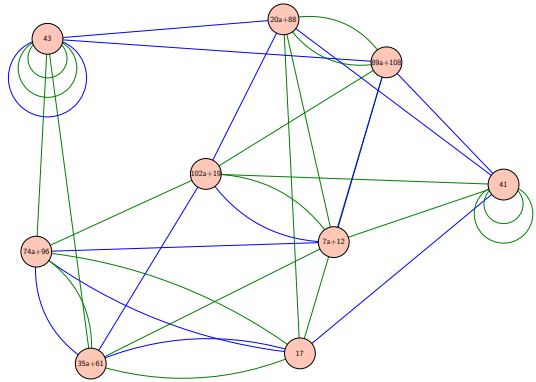

		\minipage{0.38\textwidth} 
			\centering
			\include{NEWSSIGp=109l=2,3}
			\label{HeuristicallyIndp757l35}
		\endminipage
			\caption{The graph $\Lambda_{109}(2,3)$ where $2$-isogenies are in blue and $3$-isogenies are in green.} 
		\end{figure}

\textbf{Outline.} We start with Section \ref{SectionBackground}, 
where we recall the necessary preliminaries to understand the construction of supersingular isogeny graphs and Brandt matrices, 
as well as the essential results on traces of Brandt matrices and sums of Hurwitz class numbers used in our subsequent computations. 
In Section \ref{SectionLoopsAndMultiEdges}, 
we begin the study of loops in the graph $\Lambda_p(\l)$, which is the simplest example one can consider, 
in order to demonstrate how our method of computing collisions works. 
This gives us a bound on the number of loops and well as conditions to ensure the graph $\Lambda_p(\l)$ has no loops. 
Subsequently, we turn our attention to the simplest kind of collision one can have: multi-edges. 
We then obtain bounds on the number of multi-edges and conditions to guarantee the simplicity of the graph $\Lambda_p(\l)$. 
Finally, in Section \ref{SectionSimultaneousstudyoftwographs}, we simultaneously study two supersingular isogeny graphs. 
This section provides new results regarding problems that haven’t been considered before. 
We first study the number of edges in common between the two graphs $\Lambda_p(\l_1)$ and $\Lambda_p(\l_2)$ 
and provide bounds for them, which in turn gives bounds for the edit distance between $\Lambda_p(\l_1)$ and $\Lambda_p(\l_2)$. 
We also give conditions under which these common edges do not exist. 
Second, we study a much more general notion of collision by introducing the concept of \emph{bi-route number}. 
It is a measure of how many times two vertices will have paths (which will constitute collisions) of certain lengths 
between them on both graphs $\Lambda_p(\l_1)$ and $\Lambda_p(\l_2)$. 

\begin{acknowledgements}
	I would like to thank Eyal Goren, who supervised my masters studies at McGill University.  
	I am thankful for the many fruitful discussions we had -- that lead to many of the results in this article -- 
	and the helpful advice I received during my time at McGill. 
\end{acknowledgements}

\section{Background} \label{SectionBackground}

\subsection{Elliptic curves}

	Let $p$ be a prime greater than $4$ and $k:= \F_{q}$ be the finite field of order $q$ and characteristic $p$. 
	The algebraic closure of $\F_{q}$ will be denoted by $\overline{\F}_{q}$. 
	An \emph{elliptic curve} $E$ over $k$ is a smooth projective curve of genus $1$, 
	together with a distinguished $k$-rational point $O_{E}$. 
	Equivalently, (since $\text{Char}(k) \not = 2,3$) the elliptic curve $E$ over $k$ is the projective closure of 
		the affine curve given by a \emph{short Weierstra{ss} equation} of the form 
		\begin{equation*}
			E: y^2 = x^3 + ax + b, \qquad a,b \in k 
		\end{equation*} 
	where the discriminant $\Delta := -16(4a^3 + 27 b^2)$ is non-zero. 
	Here, the distinguished point $O_E$ is the point at infinity $(0:1:0)$ on the projective closure. 
	The $j$-invariant of $E$ is $j(E) := -1728 \frac{(4a)^{3}}{\Delta}$. 
	Every $j \in k$ is the {$j$-invariant} of some elliptic curve $E_j$ over $k$. 
	Indeed, for $j$ equal to $0$ or $1728$, 
	consider  $E_0: y^2 = x^3 + b$ and $E_{1728}: y^2 = x^3 + ax$ respectively (for any choice of $a$ and $b$). 
	If $j\not =0,1728$, let $E_j: y^2 = x^3 + 3j (1728-j) x + 2j(1728 - j)^{2}$. 
	
	For every extension $K/k$, the set of $K$-rational points of $E$, 
		\[E(K) := \{(x,y) \in K \times K : y^2 = x^3 + ax + b\} \cup \{ O_E\}, \]  
	forms an abelian group where the group structure is given by rational functions. 
	An \emph{isogeny} $f$ between two elliptic curves $E_1$ and $E_2$ over $k$ is a group homomorphism 
	that is also a morphism of varieties (a regular rational function). We must have $f(O_{E_1}) = O_{E_2}$. 
	Bijective isogenies are just called \emph{isomorphisms}. 
	It is known that 
	isogenies have finite kernels
	and that non-zero isogenies are surjective. 
	
	Elliptic curves are classified up to isomorphism (over $\ol{k}$) by their $j$-invariants, 
	i.e. two elliptic curves $E_1$ and $E_2$, when viewed over $\ol{k}$, are isomorphic if and only if $j(E_1)= j(E_2)$. 
	
	Quotients of elliptic curves by finite subgroups are also elliptic curves. 
	In addition, isogenies are uniquely defined by their kernels. 
	So given a finite subgroup $L \sbst E$, $E/L$ is an elliptic curve and there exists a unique (up to isomorphism) elliptic curve $E'$  
	and a separable isogeny $\phi : E \longrightarrow E'$ such that $\ker (\phi) = L$ and $E’ \cong E/L$. 
	
	\begin{rmk}
		Although elliptic curves may in general be defined over any extension $K/k$, 
		in this work, 
		isogenies will always (unless specified explicitly) be defined over the algebraic closure $\ol{k}$. 
	\end{rmk}
	
	\begin{theorem} \label{21sep101136a}
		Let $E$ be an elliptic curve over a field $K$ (of any characteristic). 
		The automorphism group $\text{Aut}(E)$ is cyclic. 
		If $j(E)\not =0,1728$, then $\# \text{Aut}(E) = 2$. 
		If $\text{Char}(K)\not =2,3$, then $\# \text{Aut}(E_0) = 6$ and $\# \text{Aut}(E_{1728}) = 4$. 
	\end{theorem}
	
	Since every isogeny $\phi : E_1 \longrightarrow E_2$ is a rational map of curves, 
	composition with $\phi$ induces an injection of function fields 
		\begin{align*}
			\phi^{*} : \  \ol{k}(E_2) & \longrightarrow \ol{k}(E_1)  \\ 
			  f & \mapsto f \circ \phi. 
		\end{align*} 
	We say that the isogeny $\phi$ is \emph{separable} (resp. \emph{purely inseparable}) 
	if the field extension $\ol{k}(E_1)/ \phi^{*} \ol{k}(E_2)$ is separable (resp. {purely inseparable}). 
	The \emph{degree} of a non-zero isogeny $\phi : E_1 \longrightarrow E_2$ is 
	the degree of the field extension, i.e. $\deg (\phi) := [\ol{k}(E_1) : \phi^{*} \ol{k}(E_2)]$. 
	The degree is multiplicative (with respect to composition) and in the case where $\phi$ is a {separable} isogeny, 
	its degree is equal to the size of its kernel. 
	\begin{rmk}
		If $K$ is a field of characteristic $p$, which is the case in this paper, 
		then any isogeny $\phi$ can be factored into the composition of a separable isogeny $\psi$, and a purely inseparable 
		Frobenius isogeny $\pi_{p^{r}} : E \longrightarrow E^{(p^{r})}$, $(x,y) \mapsto (x^{p^r}, y^{p^r})$ 
		to get $\phi = \psi \circ \pi_{p^{r}}$. 
		In particular, if the degree of an isogeny is not divisible by $p$, then it is a separable isogeny. 
	\end{rmk} 
	
	Let $E$ be an elliptic curve over $K$. Given an integer $m \in \N$, the multiplication-by-$m$ map 
			\begin{align*}
				[m]:  \ E & \longrightarrow E \\ 
				 P & \mapsto \underbrace{P+ ...+P}_{m \text{-times}}
			\end{align*}
	is an isogeny of degree $m^{2}$.  
	This allows us to view $\Z$ as a subgroup of $\text{End}(E)$ 
	via $m \mapsto [m]$. 
	The kernel of $[m]$ is the group of $m$-torsion points $E[m]$ of $E$. 
	When $m$ is coprime to $\text{Char}(K)$, we have $E[m] \cong \Z/m\Z \times \Z/m\Z$. 
	\begin{eg} 
		Consider the elliptic curve $E : y^{2} = x^{3} + 1$ over $\C$ and the isogeny $[2] : P \mapsto P + P$. 
		Since isogenies are rational functions, the coordinates of the map $[2]$ can be expressed as a fraction of polynomials. 
		 Indeed, 
		 	\[[2](x,y) = \left(\frac{x^4-8 x}{4 (x^3+1)},\frac{y(-8 + 20 x^3 + x^6)}{8 (1 + x^3)^2}\right). \]
		Moreover, $\ker([2])=	E[2] = \{ O_E, (-1,0), (\frac{1+i \sqrt{3}}{2},0) , (\frac{1-i \sqrt{3}}{2},0) \} \cong \Z/2\Z \times \Z/2\Z$. 
	\end{eg}
	
	An extremely important property of isogenies, other than always being surjective, 
	is that for every isogeny $\phi : E \longrightarrow E’$  of degree $d$, 
	there exists a (unique) dual isogeny $\phi^{\vee}: E’ \longrightarrow E$ of degree $d$, 
	such that $\phi^{\vee} \circ \phi = \phi \circ \phi^{\vee} = [d]$. 
	We conclude with the following useful properties about the existence and the decomposition of certain isogenies. 
	See \cite{Sil09} 
	for more on this. 
	\begin{prop} \label{MissingLink}
		Let $\phi : E_1 \longrightarrow E_2$, $\psi : E_1 \longrightarrow E_3$ be two non-zero isogenies 
		and assume that $\phi$ is separable. 
		If $\ker(\phi) \sbst \ker(\psi)$, then there exists a unique isogeny $\gamma : E_2 \longrightarrow E_3$ 
		such that $\psi = \gamma \circ \phi$. 
	\end{prop}
	\begin{prop}\label{FactoringIsogenies}
		Let $\phi : E \longrightarrow E’$ be a separable isogeny defined over $\F_{q}$. 
		Then, there exists $n \in \Z$, elliptic curves $E=E_0,...,E_s=E’$ and 
		isogenies $\psi_i: E_i \longrightarrow E_{i+1}$ of prime degree (for $i=0,...,s-1$) 
		such that $\phi= \psi_{s-1} \circ ... \circ \psi_0 \circ [n]$. 
		Moreover, $\deg(\phi) = n^{2} \prod_{i=0}^{s-1} \deg(\psi_i)$. 
	\end{prop}

\subsection{Supersingular elliptic curves}

	\begin{defn} \label{21sep100959a}
		An elliptic curve $E$ is said to be \emph{supersingular} if it satisfies one of the following equivalent conditions:  
			\begin{enumerate}[(i)]
			\item For every finite extension $\F_{q^r}$ there are no points in $E(\F_{q^r})$ of order $p$, 
				i.e. $E(\F_{q^{r}})[p] =0$.  
			\item The isogeny $[p]: E \longrightarrow E$ is purely inseparable and $j(E) \in \F_{p^{2}}$. 
			\item \label{EndDefofSSIC}The endomorphism ring $\text{End}(E)$ is an order in a quaternion algebra. 
			\end{enumerate} 
	\end{defn}
	The quaternion algebra we refer to in part (\ref{EndDefofSSIC}) of Definition \ref{21sep100959a} is the quaternion algebra $\B_{p,\infty}$ 
	that we will define in section (\ref{BrandtMatrices}). 
	
	Any supersingular elliptic curve $E$, defined over $\ol{\F}_{p}$, has $j$-invariant $j(E) \in \F_{p^{2}}$ 
	and admits a presentation over $\F_{p^{2}}$. 
	This means that $E$ is isomorphic to another \ss \ elliptic curve defined over $\F_{p^{2}}$. 
	This implies the (generous) bound of  $p^{2}$ for the number of \ss \ elliptic curves over $\ol{\F}_{p}$ up to isomorphism.
	The actual number is much smaller than $p^{2}$ however. 
	For a prime $p \ge 5$, 
	the number of isomorphism classes (over $\ol{\F}_p$) of \ss \ elliptic curves is given by 
			\begin{equation} \label{NumberOfSSECs}
				\Big \lfloor \frac{p}{12} \Big \rfloor + 
									\left\{	\begin{array}{ll}
										0  & \mbox{if } p \equiv 1 \mod 12;  \\
										1  & \mbox{if } p \equiv 5 \mod 12 ; \\
										1  & \mbox{if } p \equiv 7 \mod 12  ;\\
										2  & \mbox{if } p \equiv 11 \mod 12. 
									\end{array}
									\right.
			\end{equation}
	Moreover, quotients of supersingular elliptic curves by finite subgroups are also supersingular.  
		
\subsection{Brandt matrices} \label{BrandtMatrices}
		A \emph{quaternion algebra} ${A}$ is a $4$-dimensional $\Q$-algebra of the form 
			\[{A} = \Q \oplus \Q \alpha \oplus \Q \beta \oplus \Q \alpha \beta \]
		such that $\alpha^{2} = a,\beta ^{2} = b$ with $a,b \in \Q^{\times}$, 
		and $\alpha \beta = - \beta \alpha$. 
		An \emph{order} $\O$ in a $\Q$-algebra $A$ is a subring of $A$ that is also a lattice (a finitely generated $\Z$-module) 
		spanning $A$ over $\Q$ (i.e. $\mathcal{O} \otimes_{\Z} \Q= A$). 
	Given a quaternion algebra $B$ and a prime $\nu$ (that could potentially be infinity), 
	we define $B_\nu := B \otimes_{\Q} \Q_\nu$ (with $\Q_\infty = \R$) . 
	It follows from Wedderburn's theorem that $B_\nu$ must either be a division ring or a matrix algebra $M_2(\Q_\nu)$. 
	In the case where $B_\nu$ is a division ring we say that $B$ is \emph{ramified} at $\nu$. 
	If $B_\nu \cong M_2(\Q_\nu)$ we say that $B$ is \emph{unramified} (or \emph{split}) at $\nu$. 
	
	A quaternion algebra is determined, up to isomorphism, by the set of primes at which it ramifies: 
	this set has even cardinality (and can potentially include $\infty$). 
	Conversely, any such set (of even cardinality, potentially including $\infty$) may arise as the set of primes 
	at which a quaternion algebra ramifies. 
	
	We will follow \cite{Gro87} and \cite{Voi21} to introduce Brandt matrices. 
	Let $p$ be a rational prime and $\B= \B_{p,\infty}$ the (unique) quaternion algebra over $\Q$ which is ramified at $p$ and infinity. 
	Let $O$ be a fixed maximal order of $\B$. 
	We will say that two left ideals $I$ and $J$ of $O$ are equivalent if there is some $b \in \B^{\times}$ such that $J=Ib$. 
	Let $\{ I_1, ..., I_n\}$ be a set of representatives for the left ideals of $O$ (with $I_1=O$). 
	Here, $n$ is independent of the choice of maximal order $O$ in $\B$. 
	Further, let $R_i := O_{\text{right}}(I_i) = \{ b \in \B : I b \sbst I \}$ be the right order of $I_i$. 
	Then each conjugacy class of a maximal order of $\B$ is represented (not necessarily once) in the set $\{R_1, ..., R_n\}$. 
	Let $t \le n$ be the number of conjugacy classes of maximal orders in $\B$. 
	We call $n$ the \emph{class number} of $\B$ and $t$ the \emph{type number} of $\B$. 
	\\
	
	The above shows that endomorphism rings of supersingular elliptic curves correspond 
	to maximal orders in the quaternion algebra $\B_{p,\infty}$. 
	The following theorem makes it clear. 
	\begin{theorem}[Deuring’s correspondence, \cite{Deu41}] 
		Fix any maximal order $O \in \B$. 
		Let $\mathscr{E}$ be the set of supersingular elliptic curves up to isomorphism 
		and let $\mathscr{L}_{O}$ be the set of equivalence classes of left ideals of $O$. 
		Then there is a correspondence between $\mathscr{E}$ and $\mathscr{L}_{O}$, such that for each $E \in \mathscr{E}$, 
		there exists a unique $I \in \mathscr{L}_{O}$ with $\text{End}(E) \cong O_{\text{right}}(I)$ 
		and $\text{Aut}(E) \cong O_{\text{right}}(I)^{\times}$. 
	\end{theorem}

	Let $\Gamma_i := R_i^{\times}/\Z^{\times}$. 
	It is a discrete subgroup of the compact group $(\B \otimes \R)^{\times}/\R^{\times} \cong \text{SO}_3(\R)$ 
	and hence must be finite. 
	Let $w_i:= \# \Gamma_i$, 
	then $w:= \prod_{i=1}^{n} w_i$ is independent of the choice of $O$ and is equal 
	to the denominator of $\frac{p-1}{12}$, when simplified.  
	Eichler's mass formula states that 
		\[\sum_{i=1}^{n} \frac{1}{w_i} = \frac{p-1}{12}. \] 
	
	We will now introduce the notions of theta series and Brandt matrices. 
	Let $I_{i}^{-1}:= \{ a \in \B : I_{i} a I_{i} \sbst I_{i} \}$ and $M_{ij}:= I_j^{-1}I_i = \{ \sum a_k b_k : a_k \in  I_j^{-1}, b_k \in I_i \}$. 
	Given $a \in \B$, let $\Nm(a)$ denote its reduced norm: 
		$\Nm(t+x i+y j+z k)=t^2 - \alpha x^2 - \beta y^2 + \alpha \beta z^2.$ 
	Let $\Nm(M_{ij})$ denote the unique rational number such that $\left \{ \frac{\Nm(a)}{\Nm(M_{ij})} : a \in M_{ij} \right \}$ 
	are integers with no common factors. 
	Define the theta series 
		\[\theta_{ij}(\tau) := \frac{1}{2 w_j} \sum_{a \in M_{ij}} e^{2 \pi i \frac{\Nm(a)}{\Nm(M_{ij})} \tau}
						= \sum_{m \ge 0} B_{ij}(m)q^{m}, \]	
	where $q:= e^{2 \pi i \tau}$. 
	This allows us in turn to define $B(m):= \begin{bmatrix} B_{ij}(m) \end{bmatrix}_{1 \le i,j \le n}$, 
	the \emph{Brandt matrix of degree $m$}. If $m=0$, we have 
		\[B(0) = \frac{1}{2} 
		\begin{bmatrix}
			\frac{1}{w_1} & \frac{1}{w_2} &... &\frac{1}{w_n}\\ 
			\frac{1}{w_1} & \frac{1}{w_2} &... &\frac{1}{w_n}\\ 
			\vdots & \vdots & \ddots&\vdots \\ 
			\frac{1}{w_1} & \frac{1}{w_2} &... &\frac{1}{w_n}
		\end{bmatrix}, \] 
	and $B(1)$ is the identity matrix. 
	For a reason that will become clear in Section \ref{Thebiroutenumber}, 
	we will define $B(m)$ for all $m \in \Q$, setting $B(m)=0$ if $m \not \in \N$ and as above if $m \in \N$.  
	\begin{prop}[Proposition 2.7 in \cite{Gro87}]
		\label{GrossProp2.7}
		\begin{enumerate}
		\item
		For $m \ge 1$, the matrix $B(m)$ has non-negative, integral entries. 
		Further, the row sums of $B(m)$ are independent of the chosen row and 
			\[\sum_j B_{ij}(m) = \sum_{\substack{d | m \\ (d, p)=1}} d. \]
		\item
		If $(m,m')=1$, then $B(mm')= B(m)B(m')$. 
		\item
		If $\l \not = p$ is a prime, then $B(\l^{k}) = B(\l^{k-1}) B(\l) - \l B(\l^{k-2})$ for all $k \ge 2$. 
		\item
		We have the symmetry relation $w_j B_{ij}(m) = w_i B_{ji}(m)$. 
		\end{enumerate} 
	\end{prop}

	The concepts of Brandt matrices and elliptic curves are intimately related. 
	We can thus use the theory of Brandt matrices to study elliptic curves. 
	This approach is very fruitful and we will use it in many computations 
	in Sections \ref{SectionLoopsAndMultiEdges} and \ref{SectionSimultaneousstudyoftwographs}. 
	To make this more precise, we introduce \emph{Hurwitz Class Numbers}. 
	\\ 
	
	Given an order $\O$ (of rank $2$ over $\Z$) of negative discriminant $d$, 
	let $h(d)$ be the size of the class group of $\O$ and $u(d)=\# (\O^{\times}/\Z^{\times}) = \# \O^{\times}/2$.  
	If $d=d_K$ is a discriminant
	of a field $K$, 
	by the correspondence between ideal class groups and form class groups, 
	we can define $h(d)$ as 
	the class number of primitive binary quadratic forms of discriminant $d$ 
	(equivalently, it is the number of reduced primitive binary quadratic forms of discriminant $d$). 
	We notice that $u(d)$ is always $1$; except if $d=-3,-4$, in which case, $u(d) =3,2$ respectively. 
	
	\begin{defn}
	For $D>0$, the \emph{Hurwitz Class Number} $H(D)$ is 
		\begin{equation*}
			H(D) = \sum_{\substack{d \cdot \f^{2} =-D \\ d \text{ neg. disc.}} } \frac{h(d)}{u(d)}. 
		\end{equation*} 
	where the sum runs over negative discriminants. For $D=0$, we set $H(0):= -1/12$. 
	\end{defn} 
	The sum defining $H(D)$ precisely takes into account $d$’s such that $d \equiv 0,1$ mod $4$. 
	It then follows that $H(D)>0$ when $D \equiv 0,3$ mod $4$. In particular, $H(4 k - s^{2})>0$ for all $s^{2} < 4k$. 
	\begin{theorem} [\cite{Gie80, Hur85, Kro60}]
		\label{TheoremOnMoregeneralSumofnonmodifiedHCNs}
		Let $m \in \Z$. Then, 
	\begin{equation*} 
		\sum_{\substack{s \in \Z \\ s^2 \le 4 m}} H(4m - s^{2} ) = 2 \sum_{d | m} d - \sum_{d | m} \min\{d,m/d\} . 
	\end{equation*} 
	\end{theorem} 
	In particular, for a given prime $\l$, we have   
			\begin{equation} \label{CorApr24308pm}  
				\sum_{|s| < 2 \sqrt{\l}} H(4\l - s^{2} ) = 2\l. 
			\end{equation} 
	\begin{defn} \label{Aug20720p}
	For $D>0$, the \emph{modified Hurwitz Class Number} attached to a prime $p$ is 
			\begin{equation*}
			H_p(D):= \left\{
						\begin{array}{ll}
							0  & \mbox{if } p \text{ splits in } \O_{-D};  \\
							H(D)  & \mbox{if } p \text{ is inert in } \O_{-D} \\ 
									& \text{ and does not divide the conductor of } \O_{-D}; \\
							\frac{1}{2}H(D)  & \mbox{if } p \text{ is ramified in } \O_{-D} \\ 
									& \text{ but does not divide the conductor of } \O_{-D}; \\
							H(\frac{D}{p^{2}})  & \mbox{if } p \text{ divides the conductor of } \O_{-D}; 
						\end{array}
					\right. 	
			\end{equation*} 
	where $\O_{-D}$ is \emph{the} order of discriminant $-D$. If $D=0$, let $H_p(0) := \frac{p-1}{24}$. 
	\end{defn} 
	Notice that $H_p(D) \le H(D)$ for all $p$ and $D>0$. 
	In addition, when $D \equiv 0,3$ mod $4$ we get $H(D)>0$ and then $H_p(D)=0$ if and only if $p$ splits in $\O_{-D}$. 
	Moreover, if we let $-d$ denote the fundamental discriminant of $\O_{-D}$ and $\mathfrak{f} = \sqrt{{D}/{d}}$ its conductor, 
	then $d,\mathfrak{f} \le D$. But for $p$ to ramify, it precisely needs to divide $d$. 
	So if $D < p$, $H_p(D)$ can be defined as 
			\begin{equation*} 
				H_p(D)= \left\{
						\begin{array}{ll}
							0  & \mbox{if } p \text{ splits in } \O_{-D};  \\
							H(D)  & \mbox{if } p \text{ is inert in } \O_{-D}. 
						\end{array}
					\right.
			\end{equation*} 
	Gross computes the trace of Brandt matrices in terms of sums of generalized Hurwitz Class Numbers as follows. 
	\begin{theorem}[Proposition 1.9 in \cite{Gro87}] \label{TraceViaHCNs}
		For all $m \ge 0$, we have 
			\begin{equation*} 
				\text{Tr}(B(m)) = \sum_{\substack{s \in \Z \\ s^2 \le 4 m}} H_p (4 m - s^2).
			\end{equation*} 
	\end{theorem} 
	\begin{eg}
		The number of supersingular elliptic curves over $\overline{\F}_p$ up to isomorphism is equal to the trace of $B(1)$, 
		as we will see in the next Section. 
		If we take $m=1$ in the above theorem, we get 
			\begin{align*}
				\text{Tr}(B(1)) &= \sum_{\substack{s \in \Z \\ s^2 \le 4}} H_p (4 - s^2) \\
							&= H_p(4) + 2H_p(3) + 2H_p(0) \\
							&= \frac{1}{4} \Big(1- \big( \dfrac{-4}{p} \big) \Big) 
								+  \frac{1}{3} \Big(1- \big( \dfrac{-3}{p} \big) \Big) + \frac{p-1}{12}, 
			\end{align*} 
		which proves that the number of 
		supersingular elliptic curves is indeed as described in (\ref{NumberOfSSECs}). 
	\end{eg}

\subsection{Supersingular isogeny graphs} \label{SectionSupersingularisogenygraphs}
	Supersingular isogeny graphs are graphs that arise from supersingular elliptic curves over finite fields 
	and isogenies of a given degree between them. 
	Fix two different primes $\l$ and $p$, with $p >3$. 
	In practice, one would let $\l$ be small (usually $\l \in \{2,3,5,7\}$) and let $p$ be a very large prime (of cryptographic size). 
	The \emph{supersingular $\l$-isogeny graph (of level 1)  for the prime $p$}, denoted by $\Lambda_p(\l)$, is constructed as follows. 
	The vertex set of this graph consists of the supersingular moduli (i.e. the set of $j$-invariants of supersingular elliptic curves). 
	Equivalently, one can view the vertices as isomorphism classes of supersingular elliptic curves over $\ol{\F}_{p}$. 
	The number of vertices $n$ is then given by (\ref{NumberOfSSECs}). 
	
	Defining the edges of $\Lambda_p(\l)$ is a bit more subtle. 
	For every \ss \ elliptic curve $E$ and (\emph{cyclic}) 	subgroup $C$ of $E$ of order $\l$, 
	draw an edge from $E$ to $E/C$. 
	Since all elliptic curves have $\l+1$ subgroups of order $\l$, 
	we obtain a directed graph $\Lambda_p(\l)$ of out-degree $\l+1$. 
	Equivalently, we can define the edges of $\Lambda_p(\l)$ as $\l$-isogenies up to \emph{automorphism of their images}, 
	(a.k.a. up to \emph{post-composing by automorphisms}) 
	i.e. two isogenies $\phi_1, \phi_2 :E_1 \longrightarrow E_2$ are equivalent 
	if there is some $\alpha \in \text{Aut}(E_2)$ such that $\phi_1 = \alpha \circ \phi_2$. 
		
		Recall that for every isogeny $\phi : E_1 \longrightarrow E_2$, 
		there is a dual isogeny $\phi^{\vee} : E_2 \longrightarrow E_1$ such that $\deg (\phi) = \deg (\phi^{\vee})$ and $\phi \circ \phi^{\vee} = [\deg(\phi)]$. 
		One might hence try to simplify these graphs by associating $\phi$ with $\phi^{\vee}$ and hope to obtain undirected graphs. 
		However, this is not possible in general. This is because we are considering isogenies up to automorphisms of the image, 
		i.e. up to composition with automorphisms \emph{from the left} (and not the right!). 
		Here is the prototypical example illustrating why one cannot associate $\phi$ to $\phi^{\vee}$ in general. 
		Let $f,g : E_1 \longrightarrow E_2$ be two isogenies. 
		Suppose that $f \sim g$, i.e. there is an automorphism $\alpha \in \text{Aut}(E_2)$ such that $f = \alpha \circ g$. 
		However, $f^{\vee} = (\alpha \circ g)^{\vee} = g^{\vee} \circ \alpha^{\vee}$ and we are now composing 
		with $\alpha^{\vee}$ from the right, not from the left.  
		So we do not necessarily know that $f^{\vee} \sim g^{\vee}$, as they differ by an automorphism of the domain, not the image. 
		If we wish to guarantee that $f\sim g$ if and only if $f^{\vee} \sim g^{\vee}$, 
		we would need that for all isogenies $f : E_1 \longrightarrow E_2$ and all $ \alpha \in \text{Aut}(E_1)$ 
		there exists some $\alpha’ \in \text{Aut}(E_2)$ such that $f \circ \alpha = \alpha’ \circ f.$
		The only way for this to work would be to only have trivial automorphisms: $\{\pm \text{Id} \}$. 
		However, we know from Theorem \ref{21sep101136a} that the elliptic curves $E_0$ and $E_{1728}$ both (and only them) 
		have non trivial automorphism groups. 
		One way (\emph{the} way) to to avoid these curves is to require that the prime $p$ is congruent to $1$ modulo $12$. 
		Then, $E_0$ and $E_{1728}$ would not be supersingular and all the supersingular elliptic curves will have 
		automorphism group $\{\pm \text{Id} \}$. 
		Hence, if $p \equiv 1 \mod 12$, then for all $\l$, 
		the graph $\Lambda_p(\l)$ can be viewed as an undirected graph of degree $\l+1$. 
		\begin{figure}[h!] 
		\minipage[3cm]{0.5\textwidth}
			\centering 
			\include{NEWSSIGp=73l=2}
			\caption{The directed graph $\Lambda_{73}(2)$.}
			\label{IsogGraphp73l2}
		\endminipage
		\hfill 
		\minipage{0.5\textwidth}
			\centering
			\include{NEWSSIGp=71l=2Directed}
			\caption{The undirected graph $\Lambda_{71}(2)$.}
			\label{IsogGraphp71l2}
		\endminipage
		\end{figure}
		
		These supersingular isogeny graphs are directly related to Brandt matrices. 
		Indeed, $B_{ij}(\l)$ is the number of equivalence classes of isogenies from $\phi: E_i \longrightarrow E_j$ of degree $\l$ 
		(cf. Proposition 2.3 in \cite{Gro87}). 
		Equivalently, it is the number of subgroups $C$ of $E_i$ of order $\l$ such that $E_i/C \cong E_j$. 
		Therefore, $B(\l)$ (where it is customary to \emph{not} denote the dependence on $p$) 
		is the adjacency matrix of the graph $\Lambda_p(\l)$. 
		In particular, $B(\l)$ is an $n \times n$ matrix and $\text{Tr}(B(1))= \text{Tr} (\text{Id})=n$.  
				
		An isogeny of degree $\l^{a}$ between two supersingular elliptic curves $E_i$ and $E_j$ 
		can be factored as a product of isogenies as in Proposition \ref{FactoringIsogenies}
		and thus can be seen as a path of length $a$ in $\Lambda_p(\l)$ from $E_i$ to $E_j$. 
	It is known that requiring the kernel of a map $\phi \in \text{Hom}(E_i,E_j)$ of degree $\l^{a}$ to be 
	cyclic is equivalent to requiring that the corresponding path of length $a$ on $\Lambda_p(\l)$ 
	does not involve any backtracking. 
	Indeed, factor $\phi$ as $\prod_{k} \psi_k$ (see proposition \ref{FactoringIsogenies}). 
	Backtracking can only be achieved through taking the same edge twice (once going forward and once going backward). 
	This means that we are composing some isogeny $\psi_k$ (going forward) with its dual $\psi_k^{\vee}$ (going backward). 
	And we know that $\psi_k \circ \psi_k^{\vee} =[\deg(\psi_k)] = [\l]$. 
	Thus, we see that an isogeny (a path) $\phi$ of degree $\l^{a}$ (of length $a$) involves backtracking 
	if and only if the factorization of $\psi$ involves some \emph{scalar} isogeny 
	(i.e. an isogeny of the form $[m]$ for some $m \in \Z$).  
	And the scalar isogeny $[\l]$ has a kernel isomorphic to $(\Z/\l\Z)^{2}$, 
	whereas all the other isogenies in the factorization of $\phi$ have a kernel isomorphic to $\Z/\l^{2}\Z$. 
	
	The supersingular isogeny graphs $\Lambda_p(\l)$ have many nice properties. 
	They are sparse, as are the Brandt matrices $B(\l)$. 
	They are $\l+1$ regular by part (1) of proposition \ref{GrossProp2.7}. 
	The second eigenvalue of the adjacency matrix of $\Lambda_p(\l)$ is bounded above by $\sqrt{\l}$ (see \cite{Piz90}). 
	They are also optimal expanders, which implies that a walk on $\Lambda_p(\l)$ quickly becomes close to a uniform distribution.  
	More precisely, 
	after 
	$\log(2n)/\log ({(\l+1)}/{2\sqrt{\l}} )$ steps, 
	a random walk on $\Lambda_p(\l)$ 
	approximates the uniform distribution with an error of $1/2n$ 
	(see Corollary 6 in \cite{Gol01} or Proposition 2.1 in \cite{DFJP14}). 
		
\section{Loops and multi-edges} \label{SectionLoopsAndMultiEdges}
		It is known that the \ss \ isogeny graphs $\Lambda_p(\l)$ are very close to being simple graphs. 
		We will make this precise by studying the number of loops and multi-edges that they usually contain. 
				
\subsection{The number of loops} \label{The_number_of_loops}
		In this section only, we will not assume that $p$ is necessarily congruent to $1$ modulo $12$. 
		We are interested in computing the 
		number of loops that supersingular isogeny graphs contain. 
		To find out, we need to count the number of loops at every vertex.  
		The vertices correspond to the supersingular elliptic curves $\{E_i\}_{i=1,...,n}$ and the edges between them 
		correspond to isogenies up to automorphism from the left (i.e. automorphism of  the image). Hence, we can write 
			\begin{align} \label{May27406AM}
			\begin{split}
				\sharp \{\text{loops in } \Lambda_p(\l) \} 
							&=\sum_{i=1}^{n} \# \{f \in \text{End}(E_i) :  \deg(f) = \l  \}/\# \text{Aut}(E_i) \\ 
							&= \sum_{i=1} B_{ii} (\l) \\
							&= \text{Tr}(B(\l)) \\
							&= \sum_{\substack{s \in \Z \\ s^2 \le 4 \l}} H_p (4 \l - s^2). 
			\end{split} 
			\end{align} 
		Let us look closer at the quantity $H_p (4 \l - s^2)$. 
		Since $\l$ is a prime, $4 \l - s^2 \not = 0$, and we do not have to worry about the case $H_p(0)$.
		This case will appear in section \ref{TheNumberOfMultiEdges}, 
		when dealing with the number of multiple edges. 		
		As $4 \l - s^2 \equiv 0,3$ mod $4$, we have $H(4 \l - s^2) >0$ for all $s^2 \le 4 \l$. 
		So the above computation 
		and Definition \ref{Aug20720p} 
		give the following known result. 
		\begin{theorem}[The No-Loop Theorem]
		\label{NoLoopsThm}
			The \ss \ isogeny graph $\Lambda_p(\l)$ has no loops 
			if and only if $p$ splits in $\O_{s^2 - 4 \l}$ for all $s^2 \le 4 \l$, 
			i.e. if and only if $\big( \frac{s^2 - 4 \l}{p} \big) =1$ for all $s^2 \le 4 \l$. 
		\end{theorem} 
		In most application, the prime $\l$ is very small (usually $\l \in \{2,3\}$). Once we fix $\l$, 
		we only have a finite number of orders in which we need to investigate the behaviour (splitting) of $p$. 
		Moreover, the splitting of $p$ in such orders can be phrased in terms of congruence conditions on $p$ 
		via the quadratic reciprocity.
		Let us now study the presence of loops in the graphs $\Lambda_p(2)$ and $\Lambda_p(3)$. 
		\begin{eg}\label{exampleLoopsl=2}
			If $\l=2$, we need to consider the orders $\O_{-D}$ where $-D \in \{ -8, -7, -4\}$. 
			 The \hyperref[NoLoopsThm]{No-Loop Theorem} implies that $\Lambda_p(2)$ has no loops if and only if 
				\[ 1= \Big(\frac{-8}{p} \Big) = \Big(\frac{-7}{p} \Big) =\Big(\frac{-4}{p} \Big). \] 
			Using quadratic reciprocity, 
			we obtain the congruence conditions 
			\begin{equation*} \label{Eg_l=2_1}
				\left\{
						\begin{array}{ll}
							p \equiv 1,2,4 \mod 7; \\ 
							p \equiv 1 \mod 8; 
						\end{array}
					\right.
			\end{equation*} 
		i.e. $p \equiv 1, 9, 25 \mod 56$. 
		One can require in addition that $p \equiv 1 \mod 12$ so that the graphs are undirected.  
		And so, one would get:
			\begin{equation} \label{Eg_l=2_2}
				p \equiv 1, 25, 121 \mod 168. 
			\end{equation} 
		Take for example the primes $193$ and $113$, respectively, 
		to obtain the isogeny graphs $\Lambda_{193}(2)$ and $\Lambda_{113}(2)$ 
		shown in Figures \ref{IsogGraphp193l2} and \ref{IsogGraphp457l2}. 
		These are the smallest primes such that the graph $\Lambda_p(2)$ is undirected (respectively directed) and has no loops. 
		\begin{figure}[h!]
		\minipage{0.4\textwidth} 
			\centering 
			\include{NEWSSIGp=193l=2}
			\caption{The graph $\Lambda_{193}(2)$}
			\label{IsogGraphp193l2}
		\endminipage
		\hfill 
		\minipage{0.58\textwidth}
			\centering
			\include{NEWSSIGp=113l=2}
			\caption{The graph $\Lambda_{113}(2)$}
			\label{IsogGraphp457l2}
		\endminipage
		\end{figure}
		\end{eg}		
		\begin{eg}\label{exampleLoopsl=3}
			Let $\l =3$. 
			As above, the graph $\Lambda_p(3)$ has no loops if and only if $p$ splits 
			in the orders $\O_{-D}$ where $-D \in \{ -8, -7, -4\}$. 
			Hence $\Lambda_p(3)$ has no loops if and only if 
				\begin{equation*} \label{Condition_l=3}
					1= \Big(\frac{-12}{p} \Big) = \Big(\frac{-11}{p} \Big) =\Big(\frac{-8}{p} \Big) = \Big(\frac{-3}{p} \Big).
				\end{equation*} 
			These conditions translate to congruence conditions:  
			\begin{equation*} \label{Eg_l=3_1}
						\left\{
									\begin{array}{ll}
										p \equiv 1 \mod 3; \\ 
										p \equiv 1,3 \mod 8; \\ 
										p \equiv 1,3,4,5,9 \mod 11;  
									\end{array}
							\right. \\  
			\end{equation*} 
		i.e.  $p \equiv 1, 25,  49, 67, 91, 97, 115, 163, 169, 235 \mod 264$. 
		If we further require that $p \equiv 1 \mod 12$, we then get 
			\begin{equation} \label{Eg_l=3_3}
				p \equiv 1, 25, 49, 97, 169 \mod 264.  
			\end{equation} 
		The smallest such prime is $97$. It gives the isogeny graph $\Lambda_{97}(3)$ shown in Figure \ref{IsogGraphp97l3}.
		\end{eg}
		We summarize the above examples in the following theorem. 
		\begin{theorem} \label{apr24259p}
			The graph $\Lambda_p(2)$ is undirected and has no loops 
			if and only if $p \equiv 1, 25, 121 \mod 168$. 
			Similarly, the graph $\Lambda_p(3)$ is undirected and has no loops 
			if and only if $p \equiv 1, 25, 49, 97, 169 \mod 264$.
		\end{theorem}
		
		One way to compare supersingular isogeny graphs is to overlay them on top of each other in a compatible way: 
		fix the vertices to be the supersingular $j$-invariants, 
		then draw the edges of both graphs $\Lambda_p(\l_1)$ and $\Lambda_p(\l_2)$ on the same set of vertices, 
		making sure to draw the edges of $\Lambda_p(\l_1)$ and $\Lambda_p(\l_2)$ in different colours, 
		as is done in Figure \ref{HeuristicallyIndp757l35}. 
		Denote this graph by $\Lambda_{p}(\l_1,\l_2)$. 
		One might for example use Theorem \ref{apr24259p} to search for primes such that 
		both $\Lambda_p(2)$ and $\Lambda_p(3)$ have no loops. 
		The first such prime is $1873$.
		\begin{figure}[h!]
		\minipage{0.4\textwidth} 
			\centering 
			\include{NEWSSIGp=97l=3}
			\caption{The graph $\Lambda_{97}(3)$.}
			\label{IsogGraphp97l3}
		\endminipage
		\hfill 
		\minipage{0.6\textwidth} $ $ \hspace{0.5cm}
				\centering
				\include{NEWnolabelsSSIGp=1009l=2}
				\caption{The graph $\Lambda_{1009}(2)$.}
			\label{IsogGraphp1009l2}
		\endminipage
		\end{figure}

		We conclude this section with a bound (only depending on $\l$) 
		on the total number of loops in the graph $\Lambda_{p}(\l)$. 
		Equations (\ref{May27406AM}) and (\ref{CorApr24308pm}) 
		give the following bound 
				\begin{equation} \label{Apr050923p}
					\sharp \{\text{loops in } \Lambda_p(\l) \} 
							=  \sum_{\substack{s \in \Z \\ s^2 < 4 \l}} H_p (4 \l - s^2) 
					\le \sum_{\substack{s \in \Z \\ s^2 < 4 \l}} H(4 \ell - s^2) 
					= 2 \l. 
				\end{equation} 
	
		\begin{cor} \label{ThmOnNumbOfLoops}
			The supersingular isogeny graph $\Lambda_p(\l)$ has at most $2 \l$ loops. 
		\end{cor}
		
		Note that this bound does not depend on $p$. 
		In particular, for all $p$, the graph $\Lambda_p(2)$ cannot have more than $4$ loops 
		and $\Lambda_p(3)$ cannot have more than $6$ loops. 
		\begin{rmk} \label{Jan170354p}
			As we vary the prime $p$, we expect it to split in $\O_{s^2 - 4 \ell}$ 
			half of the time and remain inert the other half of the time, by Chebotarev’s density theorem. 
			Hence, for fixed $\ell$, 
			we expected the number of loops in the graph $\Lambda_p(\l)$ to be $\l$. 
		\end{rmk}

\subsection{The number of multi-edges} \label{TheNumberOfMultiEdges}
		We will, from now on, assume that $p \equiv 1 \mod 12$, unless stated otherwise, 
		so that the graphs $\Lambda_p(\l)$ are undirected. 
		It will make more sense to discuss multi-edges in the context of undirected graphs. 
A \emph{multi-edge}, also known as a \emph{double-edge}, is a pair of distinct edges between two (not necessarily different) vertices. 
Note that we allow the two vertices to be the same. 
So in particular, two loops on the same vertex will count as a double-edge. 
Note that if we have $m$ edges between two vertices (or $m$ loops on one vertex), this will give rise to ${m \choose 2}$ multi-edges. 
Using similar methods to the previous section, one can study the number of multi-edges in the graph $\Lambda_p(\l)$ 
and identify conditions under which this graph has no multi-edges. 
Let $n := ({p-1})/{12}$ denote the number of vertices in $\Lambda_p(\l)$. 

\begin{theorem} \label{July180316AM}
	An undirected graph $\Lambda_p(\l)$ has no multi-edges if and only if $\text{Tr}(B(\l^{2})) =n$. 
\end{theorem} 
\begin{proof}
		As we can see in Figure \ref{ProofOfNoMultiEdges}, 
		multi-edges $f,g:E_i \longrightarrow E_j$ generate a non trivial (i.e. non-scalar) 
		isogeny $g^{\vee} \circ f$ of degree $\l^{2}$. 
		In addition, an isogeny of degree $\l^{2}$ must factor through two isogenies of degree $\l$. 
		Therefore, the graph $\Lambda_p(\l)$ has no multi-edges 
		if and only if, for all $i$, the only endomorphism of degree $\l^{2}$ of $E_i$ is the trivial one 
		(the multiplication by $\l$ isogeny), 
		i.e. $B_{ii}(\l^{2}) = 1$ for all $i$. 
		Since, $B_{ii}(\l^{2}) \ge 1$ for all $i$, 
		the graph $\Lambda_p(\l)$ has no multi-edges if and only if $\text{Tr}(B(\l^{2})) =n$. 
		
		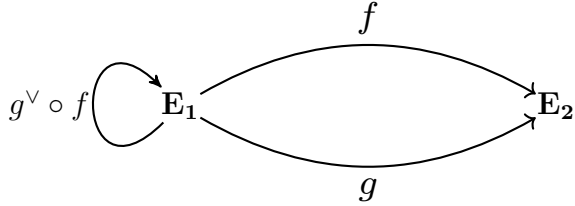
\begin{figure}[h]
			\begin{center}
			\minipage{0.6\textwidth}
			\begin{tikzpicture}[scale=2.5,auto=left,every node/.style={circle,inner sep=0pt}] 
			  \node[minimum size=0cm] (n1) at (0,0){$\mathbf{E_1}$} ;    
			  \node[minimum size=0.1cm] (n2) at (2,0){$\mathbf{E_2}$} ;       
			  \path[->,every node/.style={font=\sffamily\tiny},thick] (n1) edge[bend left] node[pos=0.5,above] 
			  			{\scalebox{1.7}{$f$}} (n2)  ;
			  \draw[->,every node/.style={font=\sffamily\tiny},thick] (n1)  edge[bend right] node[pos=0.5, below] 
			  			{\scalebox{1.7}{$g$}} (n2);
			  \Loop[dist=0.7cm,dir=WE,label=$g^{\vee} \circ f$,labelstyle=left](n1)
 			 \end{tikzpicture}  
			\caption{Multiple edges.}
			\label{ProofOfNoMultiEdges}
			\endminipage 
			\end{center}
		\end{figure} 
		
	A slightly different way to prove this theorem is to realize that a graph has no double edges if and only if 
	all the entries its adjacency matrix are either $0$ or $1$, which we write as $B_{ij}(\l)^{2} = B_{ij}(\l)$ for all $i,j$. 
	Now, as $B_{ij}(\l)^{2} - B_{ij}(\l) \ge 0$ for all $i,j$, the graph $\Lambda_p(\l)$ has no multi-edges 
	if and only if $\sum_{i,j=1}^{n} B_{ij}(\l)^{2} - B_{ij}(\l) =0$. 
		Meanwhile, by the symmetry of the Brandt matrix (since $p \equiv 1 \mod 12$) 
		and parts $(1)$ and $(3)$ of Proposition \ref{GrossProp2.7}, we have 
			\begin{equation*} 
				B_{ii}(\l^{2})=\sum_{j=1}^{n} B_{ij}(\l) B_{ji}(\l) - \l = 1+ \sum_{j=1}^{n} B_{ij}(\l)^{2} - B_{ij}(\l). 
			\end{equation*} 
	Thus, $\text{Tr}(B(\l^{2})) = \sum_{i=1}^{n} B_{ii}(\l^{2}) = n + \sum_{i,j=1}^{n} B_{ij}(\l)^{2} - B_{ij}(\l)$. 
	So $\Lambda_p(\l)$ has no multiple edges if and only if $\text{Tr}(B(\l^{2})) = n$. 
\end{proof}
\begin{rmk}
	Note that our proof also shows that the bigger the quantity $\text{Tr}(B(\l^{2})) -n$, the more we expect to have multi-edges. 
\end{rmk}
	In some cases, one might ask about the number of \emph{redundant edges}. 
	It is defined as the minimal number of edges (including loops) that one has to remove in order to have a graph with no multi-edges. 
	For example, in Figure \ref{HeuristicallyIndp757l35}, the graph $\Lambda_{109}(3)$ has $3$ redundant edges 
	and the graph $\Lambda_{109}(2)$ has no redundant edges. 
	
	As we said, the quantity $\text{Tr}(B(\l^{2})) -n$ tells us about the number of redundant edges. 
	We notice however that each redundant edge does not simply add $1$ to the value of $\text{Tr}(B(\l^{2})) -n$. 
	Actually, if we have two distinct vertices that are precisely joined by $m$ multiple edges (giving us $m-1$ redundant edges) 
	this will contribute $2m(m-1)$ to $\text{Tr}(B(\l^{2}))-n$. 
	Likewise, if we have a vertices having precisely $m$ loops (giving us $m-1$ redundant loops) 
	this will contribute $m(m-1)$ to $\text{Tr}(B(\l^{2}))-n$. 
	So for example, in the graph $\Lambda_{109}(3)$ from Figure \ref{HeuristicallyIndp757l35}, 
	we have $\text{Tr}(B(\l^{2}))-n=8=2+2+4$ where the two double loops each give a factor of $2$ and the double edge gives 
	a factor of $4$.  
	 
	Let $\text{RE}(m) := \sum_{i < j} \mathds{1}_{B_{i,j}(\l)=m} (m-1)$ be the total number of redundant edges 
	between any two distinct vertices with precisely $m$ edges between them. 
	Similarly, define $\text{RE}^{\circ}(m):= \sum_{i} \mathds{1}_{B_{i,i}(\l)=m} (m-1)$ to be the total number of redundant 
	edges on any vertex with precisely $m$ loops. 
	Then, 
		\begin{equation} \label{apr250802pm}
			\text{Tr}(B(\l^{2}))-n = \sum_{i,j=1}^{n} B_{ij}(\l) \big( B_{ij}(\l)-1 \big) 
				= \sum_{m} 2m \cdot \text{RE}(m) + m \cdot \text{RE}^{\circ}(m).
		\end{equation} 
	For $p>2$, since $\Lambda_p(\l)$ has degree $\l+1$ and must be connected, 
	there cannot be more than $\l$ edges between any two vertices. 
	Moreover, in order to have redundant edges, we must at least have $2$ edges between some given vertices. 
	Therefore, we can assume that $2 \le m \le \l$ in equation (\ref{apr250802pm}):
	\begin{equation} \label{apr250804pm}
			\text{Tr}(B(\l^{2}))-n = \sum_{m=2}^{\l} 2m \cdot \text{RE}(m) + m \cdot \text{RE}^{\circ}(m).
	\end{equation} 
	This gives the following result.  
	\begin{lem} \label{apr25757pm}
		The number of redundant edges in an undirected graph $\Lambda_p(\l)$ lies between $(\text{Tr}(B(\l^{2}))-n)/2\l$ 
		and $(\text{Tr}(B(\l^{2}))-n)/2$. 
	\end{lem}

	Let us now, as in Section \ref{The_number_of_loops}, 
	find conditions under which supersingular isogeny graphs have no multi-edges. 
	\begin{theorem}
		An undirected graph $\Lambda_p(\l)$ has no multi-edges (hence no redundant edges) 
		if and only if $p$ splits in $\O_{s^{2} - 4\l^{2}}$ for $0<s^{2} < 4\l^{2}$, 
		i.e. if and only if $\big( \frac{s^2 - 4 \l^{2}}{p} \big) =1$ for $0<s^{2} < 4\l^{2}$. 
	\end{theorem}
	\begin{proof}
		Theorem \ref{TraceViaHCNs} allows us to express the trace of $B(\l^{2})$ as a sum of Hurwits class numbers. 
		Note as well that $2  H_p(0)=(p-1)/12 =n$ 
		and that $H_p(4 \l^{2})=0$ since $p$ is always split in $\O_{-4\l^{2}}$ (because $4 \l^{2}$ is a square). 
		We can thus write 
			\begin{equation} \label{Aug030201a}
				\text{Tr}(B(\l^{2})) 	= \sum_{\substack{s \in \Z \\ s^2 \le 4 \l^{2}}} H_p (4 \ell^2 - s^2) 
								= n + \sum_{0 < s^{2} < 4 \l^{2}} H_p (4 \ell^2 - s^2). 
			\end{equation} 
		So by Theorem \ref{July180316AM}, $\Lambda_p(\l)$ has no multi-edges edges 
		if and only if $\sum_{0 < s^{2} < 4 \l^{2}} H_p (4 \ell^2 - s^2)=0$. 
		And by Definition \ref{Aug20720p}, this happens precisely when $p$ splits in $\O_{s^{2} - 4\l^{2}}$ for $s = 1,... ,2\l-1$. 
	\end{proof}

Just like in the previous section, one can find congruence conditions on $p$ 
that ensure the graph $\Lambda_p(\l)$ has no multi-edges. 
\begin{eg}\label{exampleEdgesl=2}
	Let $\l =2$ and assume that $p \equiv 1 \mod 12$. 
	Then, the graph $\Lambda_p(2)$ has no multiple edges if and only if $p$ splits in $\O_{s^2-16}$ for $s = 1,2,3$. 
	We thus want 
		\[1= \Big(\frac{-15}{p} \Big) = \Big(\frac{-12}{p} \Big) =\Big(\frac{-7}{p} \Big). \]
	Therefore, the graph $\Lambda_p(2)$ is undirected and has no multiple edges 
	if and only if $p \equiv 1, 109, 121, 169, 289, 361$ mod $420$. 
	See for instance $\Lambda_{109}(2)$ in Figure \ref{HeuristicallyIndp757l35}. 
	\end{eg}
\begin{eg} \label{apr25412p}
	Let $\l =3$ and assume that $p \equiv 1 \mod 12$. 
	Then, $\Lambda_p(3)$ has no multiple edges if and only if $p$ splits in $\O_{s^2-36}$ for $s = 1,...,5$. 
	We thus want 
		\begin{equation*} \label{ConditionsEgl=3ME}
			1= \Big(\frac{-35}{p} \Big) = \Big(\frac{-32}{p} \Big) =\Big(\frac{-27}{p} \Big)=\Big(\frac{-20}{p} \Big)=\Big(\frac{-11}{p} \Big).
		\end{equation*} 	
	So the graph $\Lambda_p(3)$ is undirected and has no multiple edges if and only if 
	$p \equiv$ 1, 169, 289, 361, 529, 841, 961, 1369, 1681, 1849, 2209, 2641, 2689, 2809, 
	3481,3529, 3721, 4321, 4489, 5041, 5329, 5569, 6169, 6241, 6889, 7561, 7681, 7921, 8089, 8761 mod $9240$. 
\end{eg}
	Putting together examples \ref{exampleLoopsl=2}, \ref{exampleLoopsl=3}, \ref{exampleEdgesl=2} and \ref{apr25412p}, 
	we obtain the following theorem. 
\begin{theorem} \label{apr25352p}
	The isogeny graph $\Lambda_p(2)$ is {\emph{undirected}} and \emph{simple} 
	if and only if $p \equiv$ 1, 121, 169, 289, 361, 529 $\mod 840$. 
	\\
	Likewise, $\Lambda_p(3)$ is \emph{undirected} and \emph{simple} if and only if $p \equiv $ 
	1, 169, 289, 361, 529, 841, 961, 1369, 1681, 1849, 2209, 2641, 2689, 2809, 
	3481,3529, 3721, 4321, 4489, 5041, 5329, 5569, 6169, 6241, 6889, 7561, 7681, 7921, 8089, 8761$\mod 9240$.
\end{theorem} 
For example, $1009$ is the smallest prime such that the graph $\Lambda_{1009}(2)$ has no loops nor multi-edges 
(see Figure \ref{IsogGraphp1009l2}).
In addition, comparing Theorem \ref{apr24259p} and examples \ref{exampleEdgesl=2} and \ref{apr25412p} gives the following 
slightly surprising corollary. 
\begin{cor}
	Let $p$ be a prime such that $\Lambda_p(3)$ is undirected and has no multi-edges. 
	Then $\Lambda_p(3)$ also has no loops (and is hence simple), 
	and $\Lambda_p(2)$ also must be simple.   
\end{cor} 
	{In particular, $2689$ is the first prime such that the supersingular $2$-isogeny and $3$-isogeny graphs both are simple.} 
	\\ 
	
	Finally, we conclude this section by giving bounds for the number of redundant edges. 
	Recall that $H(0)= -1/12$ and that $H_p(D) \le H(D)$ for all $D>0$. So, 
	by equation (\ref{Aug030201a}) and Theorem \ref{TheoremOnMoregeneralSumofnonmodifiedHCNs}, 
		\[ \text{Tr}(B(\l^{2})) -n \le \sum_{0 < s^2 < 4\l^{2}} H(4 \ell^2 - s^2) = \l+2\l^{2}- H(4 \l^{2}) +\frac{1}{6}. \] 
	But $H(4 \l^{2}) = \frac{1}{2} + h(-4\l^{2})$ 
	and 
		\begin{equation*} 
				h(-4\l^{2}) = 
									\left\{	\begin{array}{ll}
										{\l}/{2}  & \mbox{if } \l =2;  \\
										{\l}/{2} - 1/2  & \mbox{if } \l \equiv 1 \mod 4 ; \\
										{\l}/{2} + 1/2  & \mbox{if } \l \equiv 3 \mod 4. 
									\end{array}
									\right.
		\end{equation*} 
	Thus, 
		\begin{equation} \label{jan171050p}
			2\l^{2} + \frac{\l}{2} - \frac{5}{6} \le \sum_{0 < s^2 < 4\l^{2}} H(4 \ell^2 - s^2) \le 2\l^{2} + \frac{\l}{2} + \frac{1}{6}. 
		\end{equation} 
	We can then use Lemma \ref{apr25757pm} to obtain the bound $\lfloor \l^{2} + \frac{\l}{4} + \frac{1}{12} \rfloor$ 
	on the number of redundant edges in $\Lambda_p(\l)$. 
		\begin{cor} \label{ResultOnNumberOfRepeatedEdges}
			The graph $\Lambda_p(\l)$ has at most $\l^{2} + \frac{\l}{4}$ redundant edges. 
		\end{cor}
	In particular, for all $p$, the graph $\Lambda_p(2)$ cannot have 
	more than $4$ redundant edges 
	and $\Lambda_p(3)$ cannot have more than $9$ redundant edges. 
	\begin{rmk} \label{Jan180109p}
		As in Remark \ref{Jan170354p}, 
		we expect the quantity $\text{Tr}(B(\l^{2}))-n$ to be $\frac{1}{2} \sum_{0 < s^2 < 4\l^{2}} H(4 \ell^2 - s^2)$, 
		which lies between $\l^{2} + \frac{\l}{4} - \frac{5}{12}$ and $\l^{2} + \frac{\l}{4} + \frac{1}{12}$, by (\ref{jan171050p}). 
		So by Lemma \ref{apr25757pm}, 
		the expected number of redundant edges in $\Lambda_p(\l)$ lies between $\l/2$ and $(\l + 1/4)\l/2$. 
	\end{rmk}

	Corollaries \ref{ThmOnNumbOfLoops} and \ref{ResultOnNumberOfRepeatedEdges} 
	show that the supersingular isogeny graphs $\Lambda_p(\l)$ are very close to being simple, 
	as they involve a very small number of loops and redundant edges. 
	This small number of loops and redundant edges can be seen as negligible, 
	especially as the graphs that we are dealing with, in practice, are of cryptographic size.

\section{Simultaneous study of two graphs} \label{SectionSimultaneousstudyoftwographs}
		
\subsection{The intersection of two graphs}  \label{Theintersectionoftwographs}
	The \emph{intersection number} of two undirected\footnote{In the case of directed graphs, 
	we would let the sum in equation (\ref{Aug040449p}) run over all $i,j$ in $\{1,...,n\}$.} 
	graphs is the number of edges they have in common.
	It is given by 
		\begin{equation} \label{Aug040449p}
			\# \big(E(\Lambda_p(\l_1)) \cap E(\Lambda_p(\l_2)) \big) := \sum_{i \le j} \min \{ B_{ij}(\l_1) ,B_{ij}(\l_2)\}. 
		\end{equation} 
	It is related to the \emph{edit distance}, a well-known graph distance.
	By using similar methods to the previous sections, we will study 
	the number of edges that two supersingular isogeny graphs share. 
	Let $\{E_u : u=1,...,n\}$ be a set of representatives for the isomorphism classes of \ss \ elliptic curves over $\ol{\F}_{p}$. 
		Given two fixed elliptic curves $E_i,E_j$, let $\text{Hom}_{m}(E_i,E_j)$ denote the isogenies 
		of degree $m$ from $E_i$ to $E_j$. 
		Moreover, let 
			\[C_{ij}(m) := \{ f  \in \text{Hom}_{m}(E_i,E_j)  : 	\ker(f) \text{ is cyclic} \}/(\text{Aut}(E_j)), \] 
		where isogenies have cyclic kernels and are considered up to automorphisms of the image. 
		Then, $\#C_{ij}(m) = B_{ij}(m) - \# \{$isogenies from $E_i$ to $E_j$ with backtracking$\}$. 
	\begin{rmk} \label{Aug070918p}
		Since $p \equiv 1 \mod 12$, the only possible automorphisms are $\{\pm \text{Id}\}$. 
		Hence, in $C_{ij}(m)$, $f \sim g$ if and only if $f=\pm g$. 
		Moreover, it is known that isogenies are determined by their kernel up to sign. 
		Hence,  $f \sim g$ if and only if $\ker(f) = \ker(g)$.
	\end{rmk}

	\begin{theorem} 
		\label{NoEdgesinCommonTheorem}
			The graphs $\Lambda_p(\l_1)$ and $\Lambda_p(\l_2)$ have no edges in common  
			if and only if $p$ splits in $\O_{s^2 - 4 \l_1\l_2}$ for all $s^2 < 4 \l_1\l_2$, 
			i.e. if and only if $\big( \frac{s^2 - 4 \l_1\l_2}{p} \big) =1$ for all $s^2 < 4 \l_1\l_2$. 
	\end{theorem} 
	\begin{proof}
		Common edges $f \in C_{ij}(\l_1)$ and $g \in C_{ij}(\l_2)$ guarantee the existence 
		of an endomorphism $g^{\vee} \circ f$ of $E_i$ of degree $\l_1 \l_2$. 
		Since $\l_1 \l_2$ is square-free, $g^{\vee} \circ f$ has no backtracking and is in $C_{ii}(\l_1\l_2)$. 
		Conversely, an isogeny in $C_{ii}(\l_1\l_2)$ must factor as the composition of two isogenies of degrees $\l_1$ and $\l_2$, 
		giving us a pair of common edges in $\Lambda_p(\l_1) \cap \Lambda_p(\l_2)$. 
		Thus, $\Lambda_p(\l_1)$ and $\Lambda_p(\l_2)$ have no common edges if and only if $C_{ii}(\l_1 \l_2) =0$ for all $i$. 
		Note that $\#C_{ii}(\l_1\l_2) = B_{ii}(\l_1\l_2)$ since isogenies of degree $\l_1 \l_2$ cannot have backtracking.  
		So $\Lambda_p(\l_1)$ and $\Lambda_p(\l_2)$ have no common edges 
		if and only if $\text{Tr}(B(\l_1\l_2)) = 0$. 
		Finally, by Theorem \ref{TraceViaHCNs}, we express the trace of $B(\l_1\l_2)$ 
		as $\text{Tr}(B(\l_1\l_2)) = \sum_{s^{2} \le 4 \l_1\l_2} H_p(4\l_1\l_2 - s^{2})$. 
		
		Thus, by Definition \ref{Aug20720p}, $\# \big(E(\Lambda_p(\l_1)) \cap E(\Lambda_p(\l_2)) \big)=0$ if and only if 
		the prime $p$ splits in $\O_{s^2 - 4 \l_1\l_2}$ for all $s^2 < 4 \l_1\l_2$. 
		This happens if and only if $\big( \frac{s^2 - 4 \l_1\l_2}{p} \big) =1$ 
		for all $s^2 < 4 \l_1\l_2$. 
	\end{proof}

	\begin{eg}
	Let $\l_1=2$, $\l_2=3$. 
	Then $\Lambda_p(2)$ and $\Lambda_p(3)$ have no common edges 
	if and only if 
	\begin{equation*} \label{May061147}
		1=\Big( \frac{3}{p} \Big)  = \Big( \frac{-5}{p} \Big)=\Big( \frac{-8}{p} \Big)=\Big( \frac{-23}{p} \Big), 
	\end{equation*} 
	i.e. if and only if $p \equiv $ 1, 49, 121, 169, 289, 361, 409, 601, 721, 841, 961, 1129, 1369, 1681, 1729, 1849, 
	1921, 2209, 2281, 2329, 2401, 2569 mod $2760$. 
	\end{eg}

	\begin{theorem} \label{BoundOnCommonEdges}
		Let $p,\l_1,\l_2$ be primes such that $p \equiv 1 \mod 12$. 
		The two supersingular isogeny graphs $\Lambda_p(\l_1)$ and $\Lambda_p(\l_2)$  
		have at most $\l_2\l_1 + \l_2+\l_1$ edges in common.  
	\end{theorem}
	\begin{proof} 
	Assume 
	w.l.o.g. that $\l_1< \l_2$. By the definition of the intersection number, 
		\[2 \cdot \# \big(E(\Lambda_p(\l_1)) \cap E(\Lambda_p(\l_2)) \big) 
					= \sum_{i , j =1}^{n} \min \{ B_{ij}(\l_1) ,B_{ij}(\l_2)\}  + \sum_{i=1}^{n} \min \{ B_{ii}(\l_1) ,B_{ii}(\l_2)\}. \]
	First, we have $\sum_{i} \min \{ B_{ii}(\l_1) ,B_{ii}(\l_2)\} \le \text{Tr}(B(\l_1)) \le 2 \l_1$, 
	as in equation (\ref{Apr050923p}). 
	Second, we can bound $\min \{ B_{ij}(\l_1) ,B_{ij}(\l_2)\}$ above by $B_{ij}(\l_1) \cdot B_{ij}(\l_2)$. 
	Third, $\#C_{ij}(m) = B_{ij}(m)$ for all square-free $m \in \N$, 
	since backtracking involves scalar isogenies (of square degree). 
	Fourth, Lemma \ref{NoNameThm}, proven in Section \ref{Thebiroutenumber}, 
	allows us to express the trace of $B(\l_1\l_2)$ as 
		\begin{equation*}
			\text{Tr}(B(\l_1\l_2)) = \sum_{i=1}^{n} \# C_{ii}(\l_1 \l_2) 
				= \sum_{i, j=1}^{n} \#\big(C_{ij}(\l_1) \times C_{ij}(\l_2)\big) 
				= \sum_{i, j=1}^{n} B_{ij}(\l_1) \cdot B_{ij}(\l_2). 
		\end{equation*} 
	Finally, since $H_p(D) \le H(D)$ for all $D>0$, 
	Theorems \ref{TraceViaHCNs} and \ref{TheoremOnMoregeneralSumofnonmodifiedHCNs} allow us to compute 
		\begin{align*}
			\text{Tr}(B(\l_1\l_2)) 
					&= \sum_{s^{2} \le 4 \l_1\l_2} H_p(4\l_1\l_2 - s^{2}) \\
					& \le \sum_{s^{2} \le 4 \l_1\l_2} H(4\l_1\l_2 - s^{2}) \\ 
					&= 2 \sum_{d | \l_1\l_2} d - \sum_{d | \l_1\l_2} \min \left \{ d , {\l_1\l_2}/{d} \right \}  \\
					&= 2 \l_2 (\l_1+ 1). 
		\end{align*} 
	Therefore, $\# \big(E(\Lambda_p(\l_1)) \cap E(\Lambda_p(\l_2)) \big) \le \l_1\l_2 + \l_1+\l_2$. 
	\end{proof}
		
	\begin{eg} \label{Jan180123p}
		The \emph{edit distance} $\text{d}_{\text{e}}(G_1, G_2)$ 
		between two graphs $G_1=(V,E_1)$ and $G_2=(V, E_2)$ on the same set of vertices 
		is the number of edges one would need to add or remove to one graph 
		to obtain the other (see section 1.5 in \cite{lovasz2009very} for more on the edit distance). 
		It can be expressed as $\text{d}_{\text{e}}(G_1, G_2) :=  \big|  E_1  \triangle E_2 \big|$ 
		where $E_1  \triangle E_2$ is the symmetric difference.

		Let $p,\l_1,\l_2$ be primes. 
		The edit distance $\text{d}_{\text{e}}(\Lambda_p(\ell_1), \Lambda_p(\ell_2))$ 
		between $\Lambda_p(\ell_1)$ and $\Lambda_p(\ell_2)$ is given 
		by $| E(\Lambda_p(\ell_1))  \triangle E(\Lambda_p(\ell_2))  | 
				= | E(\Lambda_p(\ell_1))  | + |E(\Lambda_p(\ell_2)) | - 2 | E(\Lambda_p(\ell_1)) \cap E(\Lambda_p(\ell_2))|$. 
		However, we know that 
			\[2 |E(\Lambda_p(\ell))| = 2 \sum_{i \le j} B_{ij}(\l) 
				= \sum_{i, j=1}^{n} B_{ij}(\l) + \text{Tr}(B(\l)) =n (\l+1) + \text{Tr}(B(\l)). \] 
		So, 
			\begin{equation*}
				\text{d}_{\text{e}}(\Lambda_p(\ell_1), \Lambda_p(\ell_2)) 
					= \tfrac{1}{2} n(\l_1 + \l_2+2)+\tfrac{1}{2}\text{Tr}(B(\l_1))+\tfrac{1}{2} \text{Tr}(B(\l_2))-2 | E(G_1) \cap E(G_2) |. 
			\end{equation*} 
		And by Corollary \ref{ThmOnNumbOfLoops} and Theorem \ref{BoundOnCommonEdges} we get 
		\begin{equation}
				-2 (\l_1\l_2 + \l_1 +\l_2) 
							\le \text{d}_{\text{e}}(\Lambda_p(\ell_1), \Lambda_p(\ell_2)) - \tfrac{1}{2} n (\l_1 + \l_2+2) 
							\le \l_1 +\l_2. 
		\end{equation} 
	\end{eg} 
	\begin{rmk} 
		Theorem \ref{BoundOnCommonEdges} and Example \ref{Jan180123p} justify the intuition that in practice, 
		when $\l_1,\l_2$ are small and $p$ is big, the two supersingular isogeny graphs $\Lambda_p(\l_1)$ and $\Lambda_p(\l_2)$ 
		have very few edges in common. 
		Moreover, as in Remark \ref{Jan170354p}, we actually expect the number of edges in common 
		given in Theorem \ref{BoundOnCommonEdges} to be less than $\frac{\l_1\l_2 + \l_1+\l_2}{2}$.  
	\end{rmk}

\subsection{The bi-route number} \label{Thebiroutenumber}
	Throughout this section as well, we will assume that $p \equiv 1 \mod 12$ so that the graphs $\Lambda_p(\cdot)$ 
	are undirected. We are interested in studying the number of instances 
	where there are two curves $E_i,E_j$ that are linked by paths of length less than $R$ in both graphs. 
	\begin{defn} \label{July170858PM} 
		Given two graphs $G_1$ and $G_2$ with common vertex set $V$, their \emph{$R^{\text{th}}$ bi-route number}, 
		denoted by $I(G_1, G_2,R)$, is given by 
		\begin{equation*} 
			I(G_1, G_2;R) 
				= \sum_{a_1,a_2=1}^{R} \sum_{v,w \in V} 
						\# \left\{ \begin{matrix} 
							(\wp_1,\wp_2) \text{ such that } \wp_i \text{ is a path from } v \text{ to } 
							\\ 
							w \text{ of length } a_i  \text{ in } G_i \text{ with no backtracking}
						\end{matrix} \right \} .
		\end{equation*} 	
	In the case of the isogeny graphs $\Lambda_p(\l_1)$ and $\Lambda_p(\l_2)$, 
	the \emph{$R^{\text{th}}$ bi-route number} is denoted by $I_p(\l_1, \l_2,R)$ to alleviate the notion. 
	It is the number of \emph{collisions} $(\wp_1,\wp_2)$ in $\Lambda_p(\l_1,\l_2)$ 
	where each path $\wp_i$ has length less than $R$ and lies in $\Lambda_p(\l_i)$ with no backtracking: 
		 \begin{equation} \label{May280239AM}
		 	I_p(\l_1, \l_2,R) = \sum_{a_1,a_2=1}^{R}  \sum_{i,j=1}^{n} \#\big(C_{i,j}(\l_1^{a_1}) \times C_{i,j}(\l_2^{a_2})\big). 
		 \end{equation} 
	If the bi-route number is large, then many vertices with a path between them in $\Lambda_p(\l_1)$ 
	would also have a path between them in $\Lambda_p(\l_2)$ and \emph{vice versa}. 
	{Hence, when it comes to finding paths on the graphs, i.e. finding isogenies between elliptic curves, 
	the bi-route number gives a general idea of how similar $\Lambda_p(\l_1)$ and $\Lambda_p(\l_2)$ are}. 
	\end{defn}
	\begin{rmk} \label{Aug140832p}
		When considering the $R^{\text{th}}$ bi-route number of \emph{expander graphs}, 
		it is important to assume that $R$ is relatively small. 
		Indeed, when $R$ is big (when $R \in \Theta(\log(n))$) a random walk on the graphs $\Lambda_p(\l_i)$ 
		becomes very close to the uniform distribution. 
		So there will be many paths, 
		in both $\Lambda_p(\l_1)$ and $\Lambda_p(\l_2)$, 
		between any two vertices. 
		Therefore, for big $R$, the $R^{\text{th}}$ bi-route number will be uninformatively large and not very relevant. 
	\end{rmk}

	\begin{lem}\label{NoNameThm}
		Let $a_1,a_2 \ge 1$ 
		and suppose that $p \equiv 1$ mod $12$, 
		then $$\sum_{j=1}^{n} \#\big( C_{ij}(\l_1^{a_1}) \times C_{ij}(\l_2^{a_2})\big)  = \#C_{ii}(\l_1^{a_1}\l_2^{a_2}). $$
	\end{lem} 
	\begin{proof}
		Consider the map  
			\[\Phi: \bigcup_{j=1}^{n} \big( C_{ij}(\l_1^{a_1}) \times C_{ij}(\l_2^{a_2})\big)  
					\longrightarrow C_{ii}(\l_1^{a_1}\l_2^{a_2})) \]
		where $(f,g) \mapsto g^{\vee} \circ f$. 
		If $f$ and $g$ do not involve backtracking then $g^{\vee} \circ f$ also does not. So $\Phi$ is well defined. 
		As the above union is disjoint, it is enough to show that $\Phi$ is a bijection. 
				
		First, suppose that we have maps $f_1 \in \text{Hom}_{\l_1^{a_1}}(E_i,E_j)$, $f_2 \in \text{Hom}_{\l_1^{a_1}}(E_i,E_{j’})$, 
		$g_1 \in \text{Hom}_{\l_2^{a_2}}(E_i,E_j)$ and $g_2 \in \text{Hom}_{\l_2^{a_2}}(E_i,E_{j’})$ 
		with cyclic kernels such that $g_1^{\vee} \circ f_1 \sim g_2^{\vee} \circ f_2$, 
		i.e. $\ker(g_1^{\vee} \circ f_1) = \ker(g_2^{\vee} \circ f_2)$ by Remark \ref{Aug070918p}.  
		We already know that both $\ker(f_1)$ and $\ker(f_2)$ are isomorphic to $\Z/\l_1^{a_1}\Z$, 
		as they are cyclic and have size $\l_1^{a_1}$. But we need to show that they are equal, 
		and not just isomorphic, in order to get $f_1 \sim f_2$. 
		Let $K:= \ker(g_1^{\vee} \circ f_1) = \ker(g_2^{\vee} \circ f_2)$. It must be cyclic of degree $\l_1^{a_1}\l_2^{a_2}$, 
		so $K \cong \Z/\l_1^{a_1}\l_2^{a_2}\Z$.
		But $\ker(f_1), \ker(f_2) \sbst K$ 
		and $\Z/\l_1^{a_1}\l_2^{a_2}\Z$ has a unique subgroup isomorphic to $\Z/\l_1^{a_1}\Z$, so $\ker(f_1) = \ker(f_2)$. 
		In addition, the sequence 
		\[0 \longrightarrow \Z/\l_1^{a_1}\Z \longrightarrow \Z/\l_1^{a_1}\l_2^{a_2}\Z \longrightarrow \Z/\l_2^{a_2}\Z \longrightarrow 0\]
		splits, and we must also have $\ker(g_1) = \ker(g_2)$. 
		Thus, $\Phi$ is injective. 
		
		Now, let $h : E_i \longrightarrow E_i$ be an isogeny of degree $\l_1^{a_1}\l_2^{a_2}$ with cyclic kernel. 
		Since $p \! \! \not | \l_1^{a_1}\l_2^{a_2}$, the isogeny $h$ must be separable and $\# \ker(h) = \l_1^{a_1} \l_2^{a_2}$. 
		So $\ker(h) \cong \Z/\l_1^{a_1} \l_2^{a_2}\Z$. 
		Let $A \cong \Z/\l_1^{a_1}\Z$ be the $\l_1$-primary part of $\ker(h)$. Let $E’:= E_i/A$. 
		We get a canonical map $f:= E_i \longrightarrow E’$ of degree $\l_1^{a_1}$ with kernel $\ker(f) =A \sbst \ker(h)$. 
		Then, by Proposition \ref{MissingLink}, (see Corollary 4.11 in \cite{Sil09}), there exists a unique isogeny $g’ : E’ \longrightarrow E_i$ 
		such that $h = g’ \circ f$. Since we know that $\deg(h) = \deg(g’) \deg(f)$, we see that $\#\ker(g’) = \l_2^{a_2}$. 
		In addition, $\ker(g’) \cong \Z/\l_2^{a_2}\Z$ since $\ker(g’) = f(\ker(h)) \cong (\Z/\l_1^{a_1}\l_2^{a_2}\Z)/(\l_1^{a_1})$. 
		Hence, both kernels of $f$ and $g’$ are cyclic. 
		Finally, let $g:=(g’)^{\vee}$ to get $h= g^{\vee} \circ f =\Phi(f,g)$ as desired. 
	\end{proof}
	Lemma \ref{NoNameThm} allows us to give an alternate expression for 
	the $R^{\text{th}}$ bi-route number $I_p(\l_1, \l_2,R)$ originally defined as in equation (\ref{May280239AM}): 
	\begin{align} \label{21sep150935p}
	\begin{split}
			I_p(\l_1, \l_2,R) 
						& = \sum_{a_1,a_2=1}^{R}  \sum_{i=1}^{n} \# C_{i,i}(\l_1^{a_1} \l_2^{a_2}) \\ 
						& = \sum_{a_1,a_2=1}^{R}  \sum_{i=1}^{n} B_{i,i}(\l_1^{a_1} \l_2^{a_2})  
											- \# \{\text{maps in End}_{\l_1^{a_1} \l_2^{a_2}}(E_i) \text{ with scalar factors}\}. 
	\end{split} 
	\end{align} 
	Our goal now is to find a bound for $I_p(\l_1, \l_2,R)$. 
	We thus start by computing the number of maps in $\text{End}_{\l_1^{a_1} \l_2^{a_2}}(E_i)$ with scalar factors. 
	\begin{lem} \label{21sep140233p}
		The number of isogenies in $\text{End}_{\l_1^{a_1} \l_2^{a_2}}(E_i)$ with scalar factors 
		is $B_{i,i}(\l_1^{a_1 -2 } \l_2^{a_2}) + B_{i,i}(\l_1^{a_1} \l_2^{a_2-2}) - B_{i,i}(\l_1^{a_1 -2 } \l_2^{a_2-2})$. 
	\end{lem}
	\begin{proof} 
	Let $\phi \in \text{End}_{\l_1^{a_1} \l_2^{a_2}}(E_i)$ be an isogeny involving a scalar factor. 
	It must correspond to a path that involves backtracking, as explained at the end of Section \ref{SectionSupersingularisogenygraphs}. 
	However backtracking involves composing a factor $f$ (of minimal degree, without loss of generality) of $\phi$ 
	with its dual $f^{\vee}$ to get $[\deg(f)]$. 
	The map $\phi/ [\deg(f)]$, obtained by removing the backtracking caused by $f \circ f^{\vee}$, 
	is either in $\text{End}_{\l_1^{a_1-2} \l_2^{a_2}}(E_i)$ if the degree of $f$ was $\l_1$ 
	or in $\text{End}_{\l_1^{a_1} \l_2^{a_2-2}}(E_i)$ if the degree of $f$ was $\l_2$. 
	Therefore, by the inclusion-exclusion principle, the number of maps 
	in $\text{End}_{\l_1^{a_1} \l_2^{a_2}}(E_i)$ with scalar factors is 
	equal to $B_{i,i}(\l_1^{a_1 -2 } \l_2^{a_2}) + B_{i,i}(\l_1^{a_1} \l_2^{a_2-2}) - B_{i,i}(\l_1^{a_1 -2 } \l_2^{a_2-2})$. 
	\end{proof} 
	Lemma \ref{21sep140233p} and equation (\ref{21sep150935p}) allow us to write 
		\begin{align*}  
			I_p(\l_1, \l_2,R) &= \sum_{a_1,a_2=1}^{R} \text{Tr}(B(\l_1^{a_1} \l_2^{a_2})) -\text{Tr}(B(\l_1^{a_1-2} \l_2^{a_2})) 
						- \text{Tr}(B(\l_1^{a_1} \l_2^{a_2-2})) + \text{Tr}(B(\l_1^{a_1-2} \l_2^{a_2-2})) \\ 
						& = \text{Tr}(B(\l_1^{R} \l_2^{R}))+\text{Tr}(B(\l_1^{R-1} \l_2^{R}))
								+\text{Tr}(B(\l_1^{R} \l_2^{R-1}))+\text{Tr}(B(\l_1^{R-1} \l_2^{R-1})) \\
						& \qquad - \text{Tr}(B(\l_1^{R}))- \text{Tr}(B(\l_1^{R-1}))-\text{Tr}(B(\l_2^{R}))-\text{Tr}(B(\l_2^{R-1})) + n \\ 
						& \stackrel{(\dagger)}{=} \sum_{s^{2} \le 4 \l_1^{R}\l_2^{R} } H_p(4\l_1^{R}\l_2^{R} - s^{2}) 
								\qquad + \sum_{s^{2} \le 4 \l_1^{R}\l_2^{R-1}} H_p(4\l_1^{R}\l_2^{R-1} - s^{2}) \\ 
								& \qquad + \sum_{s^{2} \le 4 \l_1^{R-1}\l_2^{R}} H_p(4\l_1^{R-1}\l_2^{R} - s^{2}) 
								\qquad + \sum_{s^{2} \le 4 \l_1^{R-1}\l_2^{R-1}} H_p(4\l_1^{R-1}\l_2^{R-1} - s^{2}) \\ 
								& \qquad - \sum_{s^{2} \le 4 \l_1^{R}} H_p(4\l_1^{R} - s^{2}) 
								\qquad - \sum_{s^{2} \le 4 \l_1^{R-1}} H_p(4\l_1^{R-1} - s^{2})  \\ 
								& \qquad - \sum_{s^{2} \le 4 \l_2^{R}} H_p(4\l_2^{R} - s^{2}) 
								\qquad - \sum_{s^{2} \le 4 \l_2^{R-1}} H_p(4\l_2^{R-1} - s^{2}) \quad + n.
		\end{align*} 	
	Recall that $2H_p(0) = n$. 
	Moreover, on the right hand side of $(\dagger)$ in the above equation, there is precisely one instance of $2H_p(0)$ 
	in the first $4$ summations and two instances of $-2H_p(0)$ in the last $4$ summations. 
	And since 	$0 \le H_p(D) \le H(D)$ for all $D>0$, we have 
	\begin{align} \label{May020118} \allowdisplaybreaks
			\begin{split} 
				I_p(\l_1, \l_2,R) 
					& = \sum_{s^{2} < 4 \l_1^{R}\l_2^{R} } H_p(4\l_1^{R}\l_2^{R} - s^{2}) 
								\qquad + \sum_{s^{2} < 4 \l_1^{R}\l_2^{R-1}} H_p(4\l_1^{R}\l_2^{R-1} - s^{2}) \\ 
								& \qquad + \sum_{s^{2} < 4 \l_1^{R-1}\l_2^{R}} H_p(4\l_1^{R-1}\l_2^{R} - s^{2}) 
								\qquad + \sum_{s^{2} < 4 \l_1^{R-1}\l_2^{R-1}} H_p(4\l_1^{R-1}\l_2^{R-1} - s^{2}) \\ 
								& \qquad - \sum_{s^{2} < 4 \l_1^{R}} H_p(4\l_1^{R} - s^{2}) 
								\qquad - \sum_{s^{2} < 4 \l_1^{R-1}} H_p(4\l_1^{R-1} - s^{2})  \\ 
								& \qquad - \sum_{s^{2} < 4 \l_2^{R}} H_p(4\l_2^{R} - s^{2}) 
								\qquad - \sum_{s^{2} < 4 \l_2^{R-1}} H_p(4\l_2^{R-1} - s^{2}) \\ 
					&  \stackrel{(\ddagger)}{\le} \sum_{s^{2} < 4 \l_1^{R}\l_2^{R} } H(4\l_1^{R}\l_2^{R} - s^{2}) 
								+ \sum_{s^{2} < 4 \l_1^{R}\l_2^{R-1}} H(4\l_1^{R}\l_2^{R-1} - s^{2}) \\ 
								& \qquad + \sum_{s^{2} < 4 \l_1^{R-1}\l_2^{R}} H(4\l_1^{R-1}\l_2^{R} - s^{2}) 
								+ \sum_{s^{2} < 4 \l_1^{R-1}\l_2^{R-1}} H(4\l_1^{R-1}\l_2^{R-1} - s^{2}). 
			\end{split}
		\end{align} 
		Let us look closer at the last $4$ summations of Hurwitz class numbers above, in (\ref{May020118}). 
		We know from Theorem \ref{TheoremOnMoregeneralSumofnonmodifiedHCNs} that  
		\begin{align*}
			\sum_{s^{2} < 4 \l_1^{R}\l_2^{R} } H(4\l_1^{R}\l_2^{R} - s^{2})  
				& \le   2 \sum_{d| \l_1^{R}\l_2^{R}} d - \sum_{d| \l_1^{R}\l_2^{R}} \min \left(d, {\l_1^{R}\l_2^{R}}/{d} \right) \\ 
				&= \mathds{1}_{(R \text{ is even})} (\l_1 \l_2)^{R/2} +2 \sum_{\substack{d| \l_1^{R}\l_2^{R} \\ d >(\l_1\l_2)^{R/2}}} d. 
		\end{align*} 
		Hence, we can bound the $R^{\text{th}}$ bi-route number as follows. 
	\begin{theorem} \label{21sep140500p}
	Let $\l_1 < \l_2$ and $p \equiv 1 \! \mod 12$. 
	The number of \emph{collisions} $(\wp_1,\wp_2)$ in $\Lambda_p(\l_1,\l_2)$, 
	where each path $\wp_i$ has length less than $R$ and lies in $\Lambda_p(\l_i)$ with no backtracking, 
	is bounded above by  
			\begin{equation*} 
			I_p(\l_1, \l_2,R) 
				\le 	(\l_1\l_2)^{\lfloor R/2 \rfloor}  
						+ 2 \! \! \! \! \! \sum_{\substack{d| \l_1^{R}\l_2^{R} \\ d >(\l_1\l_2)^{R/2}}} \! \!  \! d  
						+ 2 \! \! \! \! \! \sum_{\substack{d| \l_1^{R-1}\l_2^{R} \\ d >\l_1^{(R-1)/2}\l_2^{R}}} \! \! \!  d 
						+ 2 \! \! \! \sum_{\substack{d| \l_1^{R}\l_2^{R-1} \\ d >\l_1^{R}\l_2^{(R-1)/2}}} \! \! \!  d
						+ 2 \! \! \! \! \! \sum_{\substack{d| \l_1^{R-1}\l_2^{R-1} \\ d >(\l_1\l_2)^{(R-1)/2}}} \! \! \!  d. 
		\end{equation*}	
	\end{theorem}

		The bound in Theorem \ref{21sep140500p} depends on $\l$ and $R$. 
		And since $R$ is never too big (see Remark \ref{Aug140832p}), 
		the bound 
		can be efficiently computed. 
		Moreover, as explained in Remark \ref{Jan170354p}, we actually expect the value of $I_p(\l_1, \l_2,R)$ 
		to be less than half of the right hand side of $(\ddagger)$ in (\ref{May020118}). 
		
		Let us now find a more concrete upper bound for $I_p(\l_1, \l_2,R)$. 
		Assume without loss of generality that $\l_1 < \l_2$, then for all $r \ge 1$, 
			\begin{equation} \label{21sep160443p}
				\sum_{\substack{d| \l_1^{r}\l_2^{s} \\ d >\l_1^{r/2}\l_2^{s/2}}} \! \! \!  d 
					= \sum_{i=0}^{r} \sum_{j= 1+ \lfloor s/2 + (r/2 -i) \log_{\l_2}(\l_1) \rfloor}^{s}  
							\! \l_1^i \l_2^j 
					\le \frac{\l_1^{r+1}\l_2^{s+1} - \l_2^{s+1} -(\l_1-1) (r+1) \l_1^{r/2}\l_2^{s/2}}{(\l_1-1)(\l_2-1)}.
			\end{equation} 
	We can then combine (\ref{21sep160443p}) with Theorem \ref{21sep140500p} to obtain the following corollary. 
	\begin{cor} 
	The $R^{\text{th}}$ bi-route number of $\Lambda_p(\l_1)$ and $\Lambda_p(\l_2)$ is bounded above by 
		\begin{equation*} 
		\begin{split} 
			I_p(\l_1, \l_2,R) 
				\le   2 (\l_1\l_2)^{R} \frac{(\l_1+1)(\l_2+1)}{(\l_1-1)(\l_2-1)} + (\l_1\l_2)^{\lfloor R/2 \rfloor} 
					- 4 \frac{\l_2^{R} (\l_2+1)}{(\l_1-1)(\l_2-1)} \\ 
					  - 2 \frac{(\l_1\l_2)^{\frac{R-1}{2}} (\sqrt{\l_1}+1)(R \sqrt{\l_2}+R+\sqrt{\l_2})}{\l_2-1}.
		\end{split} 
		\end{equation*} 
	\end{cor} 
	This gives us an idea of how many paths any two \ss \ isogeny graphs can have in common. 
	Their $R^{\text{th}}$ bi-route number $I_p(\l_1, \l_2,R)$ is in $O \big( (\l_1\l_2)^{R}\big)$.

\phantomsection 

\bibliographystyle{abbrv} 
\bibliography{ThesisReferences}

$ $

\textit{E-mail address:} \url{wissam.ghantous@maths.ox.ac.uk}
\end{document}